\documentclass[reqno, 10pt, a4paper]{amsart}

\usepackage{fullpage}
\selectfont

\usepackage[utf8]{inputenc}
\usepackage[OT2,T1]{fontenc}
\usepackage[english]{babel}

\usepackage{amsmath}
\usepackage{amsfonts}
\usepackage{amssymb}
\usepackage{amsthm}

\usepackage{mathrsfs}
\usepackage{listings}
\usepackage[mathcal]{euscript}
\usepackage{bbm}

\usepackage{algorithm}
\usepackage{algorithmic}

\usepackage{enumitem}
\usepackage{multirow}

\usepackage[dvipsnames]{xcolor}
\usepackage{graphicx}
\usepackage{tikz}
\usetikzlibrary{matrix}
\usepackage[all]{xy}

\usepackage{colortbl}
\definecolor{mygray}{gray}{0.92}

\usepackage{array}
\newcolumntype{C}[1]{>{\centering\arraybackslash$}p{#1}<{$}}

\usepackage{breqn}
\newcounter{myequation}[equation]

\usepackage{float}
\restylefloat{table}
\usepackage{multirow}

\usepackage{hyperref}

\theoremstyle{plain}
\newtheorem{theorem}{Theorem}[section]

\newtheorem{proposition}[theorem]{Proposition}
\newtheorem{lemma}[theorem]{Lemma}

\theoremstyle{definition}
\newtheorem{definition}[theorem]{Definition}

\theoremstyle{remark}
\newtheorem{remark}[theorem]{Remark}
\newtheorem{example}[theorem]{Example}

\numberwithin{equation}{section}

\def\epsilon{\varepsilon}
\def\thetac{\theta}
\def\thetar{\vartheta}

\def\Magma{\textsc{Magma} }

\def\ie{i.e. }

\def\PHM{toggle model}
\def\T0{\Theta_{0 0}}
\def\T1{\Theta_{0 1}}

\def\P{\mathbb{P}}

\def\theta{\vartheta}

\def\C{\mathbb{C}}
\def\Z{\mathbb{Z}}

\DeclareMathOperator{\Pic}{Pic}

\DeclareMathOperator{\Sp}{Sp}

\DeclareMathOperator{\Spec}{Spec}

\DeclareMathOperator{\Sym}{Sym}

\def\A{\mathbb{A}}

\def\C{\mathbb{C}}

\def\F{\mathbb{F}}

\def\O{\mathcal{O}}
\def\P{\mathbb{P}\,}
\def\Q{\mathbb{Q}}

\def\T{\mathbb{T}}
\def\Z{\mathbb{Z}}

\usepackage{bm}

\def\Oc{\mathcal{O}}
\def\Cc{\mathcal{C}}

\def\Lc{\mathcal{L}}

\newcommand{\car}[2]{%
\left[{\substack{#1\\#2}}\right]}

\hypersetup{
  pdfauthor   = {Lercier, Lorenzo Garc\'ia, Ritzenthaler},
  pdftitle    = {Reduction of non hyperelliptic genus 3 curves},
  pdfsubject  = {},
  pdfkeywords = {},
  backref=true, pagebackref=true, hyperindex=true, colorlinks=true,
  breaklinks=true, urlcolor=blue, linkcolor=blue, citecolor=blue,
  bookmarks=true, bookmarksopen=true}

\begin{document}

\title{Stable models of plane quartics with hyperelliptic reduction}
\date{\today}

\begin{abstract}
  Let $C/K : F=0$ be a smooth plane quartic over a complete discrete valuation
  field $K$. In \cite{LLLR} the authors give various characterizations of  the reduction (i.e. non-hyperelliptic genus 3 curve,
  hyperelliptic genus 3 curve or bad)  of the stable model
  of $C$: in terms of the existence of a special plane quartic model and in terms of the valuations of
  the Dixmier-Ohno invariants of $C$. The last one gives in particular an easy
  computable criterion for the reduction type. However, it does not produce a stable model, even in the case of good reduction.
   In this paper we give an algorithm to obtain (an approximation of) the stable model when the reduction of the latter is hyperelliptic and the characteristic of the residue field is not $2$. This is based on a new criterion giving the reduction type in terms of the valuations of the theta constants of $C$.  Some examples of the computation of these models are given. 
\end{abstract}


\author[Lercier]{Reynald Lercier}
\address{%
  Reynald Lercier,
  DGA \& Univ Rennes, %
  CNRS, IRMAR - UMR 6625, F-35000
 Rennes, %
  France. %
}
\email{reynald.lercier@m4x.org}

\author[Lorenzo]{Elisa Lorenzo Garc\'ia}
\address{%
	Elisa Lorenzo Garc\'ia
  Univ Rennes, CNRS, IRMAR - UMR 6625, F-35000
 Rennes, %
  France. %
}
\email{elisa.lorenzogarcia@univ-rennes1.fr}

\author[Ritzenthaler]{Christophe Ritzenthaler}
\address{%
	Christophe Ritzenthaler,
  Univ Rennes, CNRS, IRMAR - UMR 6625, F-35000
 Rennes, %
  France. %
}
\email{christophe.ritzenthaler@univ-rennes1.fr}





\maketitle



\section{Introduction and main result}

Let $K$ be a complete discrete valuation field with valuation $v$ and
valuation ring $\O$ containing a maximal ideal generated by $\pi$. Let
$k=\Oc/\langle \pi \rangle$ be the residual field. In
the following, the expression ``after a possible extension of $K$'' means that
we are allowed to take a finite
extension of $K$ and still call  $K, \O, v, \pi$ the
corresponding notions. When $F$ is an integral polynomial, \ie with
coefficients in $\O$, we denote $\bar{F}$ its
reduction modulo $\pi$.  

Given a genus 3 non-hyperelliptic curve $C/K$, in \cite{LLLR}, the authors  answer the following question: after a possible extension of $K$, what is the reduction type of the stable model $\Cc/\O$ of $C$? By this, we mean to distinguish between 
\begin{itemize}
\item $\Cc_k$ is still a non-hyperelliptic curve of genus 3. We say that $C$ has \emph{potentially good quartic reduction};
\item $\Cc_k$ is a hyperelliptic curve of genus 3. We say that $C$ has
  \emph{potentially good hyperelliptic reduction};
\item $\Cc_k$ is not a curve of genus 3. We say that  $C$ has \emph{geometrically bad reduction}.
\end{itemize}
In the first case, the special fiber is again a smooth plane
quartic over $k$,
whereas in the second case it is isomorphic over $\bar{k}$ to $y^2=f(x,z)$ where $f$
is a binary octic with no multiple roots. 

\begin{example}
	The discriminant of the Klein quartic given by $C:\,x_1^3x_2+x_2^3x_3+x_3^3x_1=0$ is equal to $7^7$, hence the model $x_1^3x_2+x_2^3x_3+x_3^3x_1=0$ over $\Z$	
	 has good reduction everywhere except at $\pi=7$. To study the reduction type of the stable model at $7$, notice that $C$ is $\bar{\Q}$-isomorphic to the curve \cite[pp.56]{ElkiesKlein}
	$$
	(x_1^2+x_2^2+x_3^2)^2+\sqrt{-7} \alpha^2 \cdot (x_1^2x_2^2+x_2^2x_3^2+x_3^2x_1^2)=0
	$$  
	with $\alpha=\frac{-1+\sqrt{-7}}{2}$. Consider now the scheme
	$$
	\mathcal{C}:\begin{cases}
	y^2=-(x_1^2x_2^2+x_2^2x_3^2+x_3^2x_1^2),\\
	\sqrt[4]{-7} \alpha \cdot y=x_1^2+x_2^2+x_3^2
	\end{cases}
	$$
	in the weighted projective space $\P^{(1,1,1,2)}$ over the ring of integers $\O$ of $K=\mathbb{Q}_7(\sqrt[4]{-7})$.  Its generic fiber is isomorphic over $K$ to $C$ whereas
	 $\mathcal{C}_k$ is isomorphic over $\bar{\F}_7$ to 
	$$\begin{cases}
	y^2=-(x_1^2x_2^2+x_2^2x_3^2+x_3^2x_1^2),\\
	0=x_1^2+x_2^2+x_3^2
	\end{cases} 
	$$
	which turns out to be the hyperelliptic curve $y^2=x^8+14x^4z^4+z^8$.
\end{example}

The shape of the stable model in  the previous example is symptomatic of the situation and motivates the following definition.

\begin{definition}\label{def:HGRM}(Def. 1.3, \cite{LLLR})
	Let $C/K$ be a smooth plane quartic. We say that $C$ admits a \emph{\PHM} if
	there exist an integer $s>0$, a primitive (\ie the gcd of its coefficients is 1) quartic form
	$G \in \O[x_1,x_2,x_3]$ and a primitive quadric $Q \in \O[x_1,x_2,x_3]$ with
	$\bar{Q}$ irreducible such that $Q^2+\pi^s G=0$ is $K$-isomorphic to $C$. If
	moreover $\bar{Q}=0$ intersects $\bar{G}=0$ transversely
	in $8$ distinct $\overline{k}$-points, we say that $C$ admits a
	\emph{good \PHM}.
\end{definition}

\begin{proposition}\label{2->1}(Prop. 1.2, \cite{LLLR}) Suppose $p \ne 2$.  Let $C/K$ be a smooth plane quartic having a good toggle model $Q^2+\pi^{2s} G=0$.
	Let us denote  $\Cc/\O$ the subscheme of the weighted projective space $\P^{(1,1,1,2)}$ defined by 
	$$\begin{cases}
	y^2 + G = 0, & \\
	\pi^s y - Q =0. &
	\end{cases}$$
	Then the generic fiber is isomorphic to the plane smooth quartic $C/K $  which has good hyperelliptic reduction. The special fiber  of $\Cc$ is isomorphic to the  double cover of  $\bar{Q}=0$  ramified over the 8 distinct intersection $\overline{k}$-points  of  $\bar{Q}=0$ with $\bar{G}=0$. 
\end{proposition}

\begin{remark}
	In \cite[Thm. 2.10]{LLLR} there is an equivalent result for the characteristic $2$ case.
\end{remark}

\begin{theorem}\label{th:goodmodel}(\cite[Thm. 1.4]{LLLR}) 
	Let  $C/K$ be a plane smooth quartic. 
	Then $C$ has good hyperelliptic reduction if and only if 
	$C$ has a good {\PHM} over $K$. 
\end{theorem}

Over $\C$, it is well-known that one can associate to the Jacobian of a genus $g$ curve $N=2^{g-1}(2^g+1)$ values, which are called theta constants. We will need a fancy version of these values over $K$ (see Section~\ref{sec:theta}) but the intuition remains the same. Now, over a DVR, we can multiply the theta constants by a common factor such that they become integral and that the minimum of their valuations is $0$. These new values $(\theta_i)_{i=0,\ldots,N-1}$ are uniquely defined up to a unit in $\O$ and we call them the \emph{integral theta constants} of the curve. In this paper we prove the following result:
\begin{theorem}\label{Thm:charThetas}
Let $C$ be a smooth plane quartic over $K$ with residue field of characteristic different from $2$. The curve $C$ has potentially good hyperelliptic reduction if and only if there is a unique integral theta constant of $C$ with positive valuation.
\end{theorem}

The direct implication of Theorem \ref{Thm:charThetas} is
relatively straightforward, for the converse we  construct an explicit
\PHM{} given by a Riemann model of the quartic and we show that it is good using the relations between the theta constants
given in \cite{weber}. The construction of the \PHM{} is completely explicit in the proof and allows us to compute a stable model in this case. We implemented this construction in \Magma~\cite{Magma} and we performed numerical experiments of some example curves.

\begin{remark}
A similar procedure could actually be realized over a number field (and would be easier to implement) when the $2$-structure of the curve is not defined over a too large extension. Unfortunately, in general, this structure is defined over an extension of degree $\# \Sp_6(\F_2)=1451520$ (\cite{harris-enu}).
\end{remark}

\begin{remark}
In \cite{LLLR}, the distinction between potentially good hyperelliptic reduction and geometrically bad reduction in terms of the Dixmier-Ohno is possible only when  $p> 7$. Our present algorithm present the advantage to allow the characteristics $p=3,5$ and $7$.
\end{remark}

\subsection*{Acknowledgement} We thank Tristan Vaccon for helpful discussions.

\section{Link with theta constants}  \label{sec:theta}

Let $R$ be a ring and $S=\Spec R$. Let $X/R$  be an abelian scheme and $\Lc$ be a relatively ample line bundle on
$X$ such that $(-1_X)^* \Lc \simeq \Lc$.  Fix an isomorphism $\epsilon: 0^* \Lc \simeq \Oc_S$ where
$0 : S \to X$ is the zero section. To any $s \in \Gamma(X,\Lc)$, Mumford associates (see \cite[Appendix
I]{mumford-tata3}) a morphism $\thetar_s : X[2] \to \A^1_S$. Following
\cite[Prop.~5.11]{mumford-tata3} (see also loc. cit. Definition.~5.8), in the special case where $S=\Spec \C$ and
$\Lc$ is the basic line bundle  on $X_{\tau}=\C^g/(\Z^g+ \tau \Z^g)$ (see
loc. cit. p.~36), then $s$ is uniquely defined up to a multiplicative constant  and there is a unique choice of $\epsilon$ such that
\begin{equation} \label{eq:theta}
\thetar_s(x) =\thetar^{\alpha}  \car{x_1}{x_2}(\tau) = e^{-i \pi x_1.x_2} \cdot \thetar \car{x_1}{x_2}(\tau)
\end{equation} 
for any $x \in X[n] \simeq (\frac{1}{n} \Z/\Z)^{2g} \ni (x_1,x_2)$ (after a
specific isomorphism of the $n$-torsion) where $\thetar \car{x_1}{x_2}(\tau)$
is the value at $0$ of the classical theta function with characteristic
$\car{x_1}{x_2}$ \cite[p.192]{mumford-tata1}.

Let $\Cc/R$ be a proper and flat scheme whose fibers are smooth curves of
genus $g>0$. Then $Pic^0(\Cc)/R$ is an abelian scheme and the previous theory
can be applied. Let $D_0$ be a theta characteristic divisor on $\Cc$.  Recall
that a theta characteristic divisor $D$ is a divisor such that
$2 D \sim \kappa$, where $\kappa$ is the canonical divisor (relative to
$S$). The divisor $ \Sym^{g-1} \Cc - D_0 \subset \Pic^0(\Cc)$ is symmetric and defines a symmetric line bundle $\Lc$ on
$\Pic^0(\Cc)$ which induces the canonical principal polarization on
$\Pic^0(\Cc)$. In particular there exists a unique (up to a multiplicative
constant) $s \in \Gamma(\Pic^0(\Cc),\Lc)$.

Let $R=\mathbbm{k}$ where $\mathbbm{k}$ is a field of characteristic
different from $2$ and for a divisor $D$ on $\Cc/\mathbbm{k}$, let us denote $L(D)$ its
Riemann-Roch space. Mumford showed in \cite{mumford-charac} that the
function $e_* : D \mapsto \dim L(D) \pmod{2}$ is a quadratic function on the
set of theta characteristic divisors $D$.  We say that $D$ is even (resp. odd) if $e_*(D)=0$
(resp. $1$). Since the map $D \mapsto D-D_0$ is a bijection between the set of
theta characteristic divisors and the set of $2$-torsion points, one can say
that the corresponding $\thetar_D= \thetar_s(D-D_0)$ is even (resp. odd). The
study of Arf invariants of quadratic forms shows that there are $2^{g-1}(2+1)$
even (resp. $2^{g-1}(2-1)$ odd) theta characteristic divisors and the
corresponding $\thetar_D$ are called \emph{theta constants} -- or
Thetanullwerte--.  As stated in loc. cit.
p.182 and refined in \cite{kempf}, the classical Riemann singularity theorem
over $\C$ extends to the present setting and in particular we get that
$\thetar_D=0$ if and only if $\dim L(D)>0$ (the odd $\thetar_D$ are therefore always equal to
zero).

In the particular case where  $\Cc/\mathbbm{k}$ is a curve of genus $3$, Clifford's theorem
\cite[IV.Th.5.4]{Hart} shows that $\dim L(D) \leq 2$, the equality being
possible only when $C$ is hyperelliptic and for a unique even theta characteristic
divisor. Hence one recovers the classical result that $C$ has $36$
(resp. 35) non-zero theta constants when $C$ is non-hyperelliptic
(resp. hyperelliptic).

Coming back to the case where $R=\O$ is a DVR and $C/K$ a non-hyperelliptic genus $3$
we denote (after a possible extension of $K$) $(\thetac_0,\ldots,\thetac_{35}) \in
\O^{36}$ its integral theta constants.  

\begin{proof}(of Theorem \ref{Thm:charThetas})
  Let us assume that $C$ has potentially good hyperelliptic reduction and let
  $\Cc/\O$ be a a smooth model of $C$ such that $\Cc_k$ is
  hyperelliptic. Since $\Pic^0(\Cc)/R$ is an abelian scheme, we can use the
  algebraic thetas $\thetar_s(x) \in \Oc$ defined above. Mumford \cite[Appendix
I]{mumford-tata3} has proved
  that these values are equal up to a constant $\lambda \in K^*$ to the
  $\thetar_D$ on $C$, therefore there are exactly $36$ of them which are
  non-zero. Moreover they shall also reduce to the theta constants
  on $\Cc_k$ which is hyperelliptic of genus $3$ and
  therefore exactly one of the $36$ non-zero ones has positive valuation. 
   Therefore after ordering them, the non-zero
  $\thetar_s(x)$ coincides  up to a unit in $\O$ with
  $(\thetac_0,\ldots,\thetac_{35})$. We therefore get that exactly one of the
  latter has positive valuation.

In the opposite direction, we assume now that among the $(\thetac_i)_{i=0,\ldots,35}$, there is a unique one with positive valuation. We are going to use the beautiful work of Weber \cite{weber} to obtain a good \PHM. Although this work is of course written  over $\C$ in the language of classical theta functions, the link \eqref{eq:theta} between classical theta constants and algebraic theta constants shows that all the algebraic homogeneous relations can be used with the latter ones. 

Weber (\cite[p.108]{weber}, see also \cite{Fiorentino}) introduces $9$ values
$(a_{ij})_{i,j=1,2,3}$ \footnote{In Weber's notation  $a_{1i}=a_i$, $a_{2i}=a'_i$ and
  $a_{3i}=a''_i$.} given as homogeneous quotient of products of theta
constants. In Section~\ref{sec:bitangents}, we will see the relations of these constants with the geometry of the curve (through its bitangents) but we will not need them right now.
We will assume that the theta constant denoted
$\thetar \car{100}{001}$ in Weber's book is the one which corresponds to the
one with positive valuation (this choice can be made without loss of
generality after a choice of a right symplectic basis of the $2$-torsion for
the Weil pairing). The expressions of the $a_{ij}$ in terms of the theta
constants imply that only $a_{11}$ and $a_{31}$ have positive valuation. Weber
gives then an explicit construction of a quartic form
$F \in K[x_1,x_2,x_3]$ such that $F=0$ is isomorphic to $C$ from the $a_{ij}$. This form is
  \begin{equation}\label{eq:Riemann}
  F=  (x_1 u_1 + x_2 u_2-x_3 u_3)^2- 4 x_1 u_1 x_2 u_2
  \end{equation}
where the $u_i$ are linear forms, solutions of the linear system
	\begin{equation}\label{eq:u}
	\begin{bmatrix}
	1 & 1 & 1 \\
	\frac{1}{a_{11}} & \frac{1}{a_{12}} &\frac{1}{ a_{13}} \\
	\frac{1}{a_{21}} & \frac{1}{a_{22}} & \frac{1}{a_{23}}
	\end{bmatrix}
	\cdot
	\begin{bmatrix}
	u_1 \\ u_2 \\ u_3
	\end{bmatrix} =
	\begin{bmatrix}
	1 & 1 & 1 \\
	a_{11} & a_{12} & a_{13} \\
	a_{21} & a_{22} & a_{23}
	\end{bmatrix}
	\cdot
	\begin{bmatrix}
	x_1 \\ x_2 \\ x_3
	\end{bmatrix}.
	\end{equation}
Let us denote $Q=x_1 u_1 + x_2 u_2-x_3 u_3$ and $G=- 4 x_1 u_1 x_2 u_2$. A series of computations (see Section~\ref{Sec:Weber}) using relations between the theta constants shows that
\begin{enumerate}
\item the $u_i$ are defined over $\Oc$ (Lemma~\ref{nonnegv});
\item $G=\pi^s G_0$ for $s >0$ an integer, $G_0$ integral and primitive and $\bar{G_0}=0$ intersects $\bar{Q}=0$ in 8 distinct points (Lemma~\ref{lem:8inter});
\item $\bar{Q}$ is a non-degenerate quadric (end of Lemma~\ref{lem:8inter}).
\end{enumerate}
We therefore get that $F$ is a good \PHM{} of $C$ and so from Proposition~\ref{2->1} we get that $C$ has potentially good hyperelliptic reduction.
\end{proof}




\section{A Riemann model providing a good \PHM}\label{Sec:Weber}

We resume with the notation of previous section. In the sequel, we
 denote $a_{11} = \pi^{v_0} a_{11}'$ and
$a_{31} = \pi^{v_0} a_{31}'$ with $v_0 > 0$ an integer and $a_{11}'$ and
 $a_{31}'$ units. If $\ell=a x_1+b x_2 + c x_3$ is a
linear form over $K$, we denote $v(\ell)=\min(v(a),v(b),v(c))$. In particular
for the linear form $u_i$ defined in \eqref{eq:u}, we let $v_i=v(u_i)$ for $i=1,2,3$.
	
	\begin{lemma}\label{nonnegv} The valuations $v_i$ are non-negative.
	\end{lemma}
	
	\begin{proof} We already know that $v(a_{ij}) \geq 0$ and $v(x_i)=0$,
          so it is enough  to check that the valuations of
          all the entries of the inverse of the matrix
		$$
		M=\begin{bmatrix}
		1 & 1 & 1 \\
		\frac{1}{a_{11}} & \frac{1}{a_{12}} &\frac{1}{ a_{13}} \\
		\frac{1}{a_{21}} & \frac{1}{a_{22}} & \frac{1}{a_{23}}
		\end{bmatrix}
		$$
		are also non-negative. We compute the inverse of $M$ by computing the adjoint matrix and dividing by the determinant of $M$.
		
		The determinant of $M$ is equal to 
		$$
		\frac{-a_{11}a_{12}a_{21}a_{23} + a_{11}a_{12}a_{22}a_{23} + a_{11}a_{13}a_{21}a_{22} - a_{11}a_{13}a_{22}a_{23} -
			a_{12}a_{13}a_{21}a_{22} + a_{12}a_{13}a_{21}a_{23}}{a_{11}a_{12}a_{13}a_{21}a_{22}a_{23}}=
		$$
		$$
		=\frac{a_{11}(-a_{12}a_{21}a_{23} + a_{12}a_{22}a_{23} + a_{13}a_{21}a_{22} - a_{13}a_{22}a_{23}) -
			a_{12}a_{13}a_{21}(a_{22} -a_{23})}{a_{11}a_{12}a_{13}a_{21}a_{22}a_{23}}
		$$
		where the factor $a_{22} -a_{23}$ equals
		$$ i\frac{\thetar\car{000}{100}\thetar\car{110}{111}\thetar\car{001}{000}\thetar\car{111}{011}}{\thetar\car{101}{111}\thetar\car{011}{100}\thetar\car{111}{101}\thetar\car{001}{110}}$$ 
		by \cite[16.14, pp.110]{weber}. Hence, it has zero valuation.
		Therefore, $v(\operatorname{det}M)=-v_0$. On the other hand, the valuation of the entries of the adjoint matrix are greater or equal to $-v_0$. So, the result follows. 
	\end{proof}
	
        If we multiply both sides of eq.~\eqref{eq:u} by $a_{11}$ and we
        look at the valuation of the entries in the equation, we
        get
\begin{equation*}
\begin{bmatrix}
v_0 & v_0 & v_0 \\
0 & v_0 &v_0 \\
v_0 & v_0 & v_0
\end{bmatrix}
\,,\ 
\begin{bmatrix}
v_1 \\ v_2 \\ v_3
\end{bmatrix} \text{for the LHS, and}
\begin{bmatrix}
v_0 & v_0 & v_0 \\
2v_0 & v_0 & v_0 \\
v_0 & v_0 & v_0
\end{bmatrix}
\,,\,
\begin{bmatrix}
0 \\ 0 \\ 0
\end{bmatrix}\text{for the RHS,}
\end{equation*}
%
%
from which we easily read that $v_1\geq v_0$ since $v_2,\,v_3$ are non-negative by Lemma~\ref{nonnegv}. We can write $u_1=\pi^{v_1}u'_1$.
This proves that the Riemann model 
$$
F=\underbrace{(\pi^{v_1}x_1u'_1+x_2u_2-x_3u_3)^2}_{Q}-4\pi^{v_1}\underbrace{x_1u'_1x_2u_2}_{G}=0,
$$
is defined over $\O$. We are going to show that it is a good \PHM.

\begin{lemma} \label{lem:8inter}
The intersection of $\bar{Q}=0$ with $\bar{G}=0$ is transverse  and the quadratic form $\bar{Q}=\,x_2\bar{u}_2-x_3\bar{u}_3$ is non-degenerate.
\end{lemma}
\begin{proof} 

Let us write down the reduction of eq.~\eqref{eq:u} before studying the intersection points (here $\epsilon=0$ if $v_1>v_0$ and 1 otherwise),
\begin{align}
\bar{u}_2 + \bar{u}_3 &= x_1 + x_2 + x_3\,, \label{eq:ru1}\\
\epsilon \frac{\bar{u}'_1}{a'_{11}} + \frac{\bar{u}_2}{a_{12}} + \frac{\bar{u}_3}{a_{13}} &= a_{12} x_2 + a_{13} x_3\,, \label{eq:ru2}\\
\frac{\bar{u}_2}{a_{22}} + \frac{\bar{u}_3}{a_{23}} &= a_{21} x_1 + a_{22} x_2 + a_{23} x_3\,. \label{eq:ru3}
\end{align}
One can also add the following equation (see \cite[Prop. 2]{Fiorentino}):
\begin{equation*}
\frac{u_1}{a_{31}} + \frac{u_2}{a_{32}} + \frac{u_3}{a_{33}} = a_{31}x_1+a_{32} x_2 + a_{33} x_3, 
\end{equation*}
which reduces to 
\begin{equation}
\epsilon \frac{\bar{u}_1}{a'_{31}} + \frac{\bar{u}_2}{a_{32}} + \frac{\bar{u}_3}{a_{33}} = a_{32} x_2 + a_{33} x_3. \label{eq:ru4}
\end{equation}

We compute the intersection of $\bar{Q}=0$ successively with $x_2=0, x_1=0, \bar{u}_2=0$ and then $\bar{u}_1'=0$. 

$\bullet$ { Intersection with $x_2=0$}.
$$
\{\bar{Q}=0\} \cap \{ x_2 = 0 \} = \{ x_3 \bar{u}_3 = 0\} \cap \{x_2 = 0 \} = 
\{ P_1 = (1:0:0)\} \cup (\{ \bar{u}_3=0\} \cap \{x_2=0\}).
$$
To compute the second intersection point above, denoted $P_2$, we see from equations~\eqref{eq:ru1} and \eqref{eq:ru3} that
\begin{equation}\label{eq:ubis3}
\left(\frac{1}{a_{23}} - \frac{1}{a_{22}}\right) u_3 = \left(a_{21} - \frac{1}{a_{22}}\right) x_1 + \left(a_{22} - \frac{1}{a_{22}}\right) x_2 + \left(a_{23} - \frac{1}{a_{22}}\right)x_3.
\end{equation}
Letting $x_2=0$ and $\bar{u}_3=0$ we need to look at the valuation of $a_{21} - {1}/{a_{22}}$.  As in Lemma~\ref{nonnegv}, Weber
\cite[16.11, 16.14]{weber} gives an expression for this difference in terms of the theta constants 
$$ (a_{21}a_{22}-1)(a_{22})^{-1}=-\frac{\thetar\car{110}{000}\thetar\car{000}{011}\thetar\car{000}{100}\thetar\car{110}{111}}{\thetar\car{011}{011}\thetar\car{101}{000}\thetar\car{101}{111}\thetar\car{011}{100}}(-i)\frac{\thetar\car{101}{111}\thetar\car{011}{100}}{\thetar\car{110}{001}\thetar\car{000}{010}}=i\frac{\thetar\car{110}{000}\thetar\car{000}{011}\thetar\car{000}{100}\thetar\car{110}{111}}{\thetar\car{011}{011}\thetar\car{101}{000}\thetar\car{110}{001}\thetar\car{000}{010}}$$
 and one can check that the valuation is $0$. Hence we get   $P_2 = \displaystyle\left(\frac{a_{23}-{1}/{a_{22}}}{{1}/{a_{22}}-a_{21}}:0:1\right)$.

$\bullet$ { Intersection with $x_1=0$}.  We  deal
with $\{\bar{Q}=0\} \cap \{ x_1 = 0 \}$ in the same way. We use
eq.~\eqref{eq:ubis3} in conjunction  with
\begin{equation}\label{eq:ubis2}
\left(\frac{1}{a_{22}} - \frac{1}{a_{23}}\right) u_2 = \left(a_{21} - \frac{1}{a_{23}}\right) x_1 + \left(a_{22} - \frac{1}{a_{23}}\right) x_2 + \left(a_{23} - \frac{1}{a_{23}}\right) x_3,
\end{equation}
and $x_1=0$ to get that the coordinates of the intersection points must satisfy 
$$
a_{22} x_{2}^2
+ (a_{22} + a_{23}) x_2 x_3 + a_{23} x_{3}^2=0.
$$
The discriminant of this quadratic form has valuation $0$ since $v(a_{22} - a_{23})=0$, still using \cite[16.14]{weber}.
We have that $-1$ and $- \frac{a_{23}}{a_{22}}$ are the two different roots of this polynomial with $x_3 = 1$, and we get that 
$$
\{\bar{Q}=0\} \cap \{ x_1 = 0 \} = \{ P_3 = (0:-1:1),\, P_4 = (0 : - \frac{a_{23}}{a_{22}} : 1) \}.
$$ 

$\bullet$ { Intersection with $\bar{u}_2=0$}. The case $\{\bar{Q}=0\} \cap \{ \bar{u}_2 = 0 \}$ is also similar.  Thanks to  \cite[16.14]{weber}, $v(a_{23}a_{22}-1)=v(a_{23}a_{21}-1)=0$ and so if $x_3 = 0$, eq.~\eqref{eq:ubis2} gives the intersection point
$$
P_5 = \left(1: - \frac{a_{23}a_{22}-1}{a_{23}a_{21}-1}:0\right).
$$
If $\bar{u}_3 = 0$, we use equations~\eqref{eq:ru1} and \eqref{eq:ru3} to compute a second intersection  point
$$
P_6 = (a_{22} - a_{23}:a_{23} - a_{21}:a_{21} - a_{22}).
$$

$\bullet$ {Intersection with $\bar{u}'_1=0$}. The case
$\{\bar{Q}=0\} \cap \{ \bar{u}'_1 = 0 \}$ is a bit more delicate as we miss
some relations from Weber which we  have to work out. Letting $x_3=1$ we can write $\bar{u}_3=x_2\bar{u}_2$. We use
this to write $\bar{u}_2$ in terms of $x_2$ in eq.~\eqref{eq:ru4} and in
eq.~\eqref{eq:ru2}. Equating  both expressions
for $\bar{u}_2$, we get a degree two equation for $x_2$:
$$
(a_{12}a_{13}-a_{32}a_{33})\left(\frac{1}{a_{13}a_{33}}x_2^2+\left(\frac{1}{a_{12}a_{33}}+\frac{1}{a_{13}a_{32}}\right)x_2+\frac{1}{a_{12}a_{32}}\right)=0.
$$
We first need to prove that it indeed defines a degree $2$ equation, that is, that the leading coefficient $a_{12}a_{13}-a_{32}a_{33}$ is not zero. Secondly, in order to prove that we obtain two different solutions for $x_2$, we need to prove that the discriminant $\left(\frac{a_{13}a_{32}-a_{12}a_{33}}{a_{12}a_{13}a_{32}a_{33}}\right)^2$ is also different from zero.

 We  first prove that the valuation of the leading coefficient is $0$.
By \cite[16.14]{weber},
$$
a_{12}a_{13}-a_{32}a_{33}=\frac{\thetar\car{000}{110}\thetar\car{010}{100}}{\thetar\car{111}{000}\thetar\car{101}{010}}
\left(\frac{\thetar\car{010}{001}\thetar\car{000}{011}}{\thetar\car{101}{111}\thetar\car{111}{101}}-\frac{\thetar\car{100}{010}\thetar\car{110}{000}}{\thetar\car{011}{100}\thetar\car{001}{110}}\right).
$$
All the theta constants involved in the formula are units, so we only need to check that
$$
D=\thetar\car{010}{001}\thetar\car{000}{011}\thetar\car{011}{100}\thetar\car{001}{110}-\thetar\car{100}{010}\thetar\car{110}{000}\thetar\car{101}{111}\thetar\car{111}{101}
$$
is also a unit. We claim that 
$$
D=\pm\thetar\car{011}{000}\thetar\car{001}{010}\thetar\car{010}{101}\thetar\car{000}{111}
$$
and so we get the result. The proof of the claim is a well-known game with (classical) theta constants that we will now play. By
 \cite[Chap. II, Theorem 18]{rauch} with the notation from there $$\car{\epsilon(1)}{\epsilon'(1)}=\car{010}{001},  \car{\epsilon(2)}{\epsilon'(2)}=\car{100}{010},  \car{\epsilon(3)}{\epsilon'(3)}=\car{011}{000},  \car{\lambda}{\lambda'}=\car{010}{010} \; \textrm{and} \; \car{\mu}{\mu'}=\car{001}{101}$$  one gets
\small
\begin{align} 
\thetar\car{010}{001}\thetar\car{000}{011}\thetar\car{011}{100}\thetar\car{001}{110}\pm\thetar\car{100}{010}\thetar\car{110}{000}\thetar\car{101}{111}\thetar\car{111}{101}\nonumber \\
\pm\thetar\car{011}{000}\thetar\car{001}{010}\thetar\car{010}{101}\thetar\car{000}{111}=0.\label{eq:3thetas}
\end{align}
\normalsize We need to prove that the first sign is negative to get our expression of $D$. We
consider this relation for Riemann matrices\footnote{using a diagonal matrix $\tau$
  we would have only got  that the second sign is negative, which is not what we want.}
$\tau=\begin{pmatrix}* & 0 & *\\ 0 & * & 0\\ * & 0 & *\end{pmatrix}$. By
\cite[Chap. I, Theorem 11]{rauch} the equation~\eqref{eq:3thetas} simplifies
to 
$$
\thetar^2\car{1}{0}\thetar^2\car{0}{1}\left( \thetar^2\car{00}{01}\thetar^2\car{01}{10}\pm\thetar^2\car{10}{00}\thetar^2\car{11}{11}\pm\thetar^2\car{01}{00}\thetar^2\car{00}{11}\right) =0.
$$
Now we use \cite[Chap. I, Theorem 5]{rauch} with the notation there $$\car{\epsilon(1)}{\epsilon'(1)}=\car{01}{10}, \car{\epsilon(2)}{\epsilon'(2)}=\car{10}{00}, \car{\epsilon(3)}{\epsilon'(3)}=\car{01}{00}, \; \textrm{and} \;  \car{\mu}{\mu'}=\car{10}{01}$$to rewrite the expression inside the parentheses above as
$$
\thetar^2\car{00}{01}\thetar^2\car{11}{11}\pm\thetar^2\car{20}{01}\thetar^2\car{11}{11}\pm\thetar^2\car{11}{01}\thetar^2\car{00}{11}=0.
$$
The third term is zero since the characteristic $\car{11}{01}$ is odd. Finally, since by \cite[Chap. I, Theorem 3]{rauch} $
\thetar\car{20}{01}=-\thetar\car{00}{01}$, the previous equation can be realized only if the first sign is a minus sign. This proves our claim on $D$ and we get a degree $2$ equation for $x_2$.

The same arguments work for proving that the discriminant has valuation
$0$. Hence, the intersection $\{\bar{Q}=0\} \cap \{ \bar{u}'_1 = 0 \}$ consists of two
distinct points $\{ P_7,\,P_8 \}$.



$\bullet$ The intersection is transverse. We need to prove that $P_i \neq P_j$ for all distinct $i,j\in\{1,...,8\}$. Recall that:
$$
P_1=(1:0:0),\,P_2=(a_{23}-{1}/{a_{22}}:0:1),\,P_3=(0:-1:1),\,P_4=(0:-a_{23}:a_{22}),
$$ 
$$
P_5=(a_{23}a_{21}-1:1-a_{23}a_{22}:0),\,P_6=(a_{22} - a_{23}:a_{23} - a_{21}:a_{21} - a_{22}),\,P_7,\,P_8.
$$
The six first point are clearly distinct, and $P_7$ and $P_8$ are distinct. We will check now that they are also distinct from the first six points. For that we  study the following intersections:

$\{\bar{Q}=0\} \cap \{ \bar{u}'_1 = 0 \} \cap \{x_2 = 0\}$: we have here
$\bar{u}'_1 = x_2 = 0$ and $x_3 = 1$, hence $\bar{u}_3 =0$ from $\bar{Q}=0$
and from equations~\eqref{eq:ru2} and \eqref{eq:ru4} we get
$a_{12}a_{13}-a_{32}a_{33}=0$, which we have just seen is not. Hence, the
intersection in empty and the points $P_7$ and $P_8$ are distinct from the
points $P_1$ and $P_2$.

$\{\bar{Q}=0\} \cap \{ \bar{u}'_1 = 0 \} \cap \{x_2 = 0\}$: we get
$x_1=\bar{u}'_1=0$, $x_3=1$ and $u_3=x_2u_2$. From eq.~\eqref{eq:ru3} we get
$\bar{u}_3=x_2$ and hence equations~\eqref{eq:ru1} and \eqref{eq:ru3} implies
$a_{13}a_{32}-a_{12}a_{33}=0$. We have also already seen that this is not
possible. Therefore, the points $P_7,\,P_8$ are distinct from the points
$P_3,\,P_4$.

$\{\bar{Q}=0\} \cap \{ \bar{u}'_1 = 0 \} \cap \{\bar{u}_2 = 0\}$: in this situation we get $\bar{u}'_1=\bar{u}_2=\bar{u}_3=0$, which is not possible since those forms are linearly independent. We then conclude that $P_7,\,P_8$ are also distinct from $P_5$ and $P_6$. 

$\bullet$ {The conic $\bar{Q}=0$ is non-singular}.  If it were, $\bar{Q}$ would be the product of two
linear forms $L_1$ and $L_2$. None of these lines $L_i=0$ is equal to
$x_1 = 0,\,x_2 = 0,\,\bar{u}'_1 = 0$ or $\bar{u}_2 = 0$, since the
intersection of the conic $\bar{Q}=0$ with each of them is two distinct
points. Moreover, $L_1$ and $L_2$ define distinct lines. Hence, we have that
$4$ of the points $P_1,...,P_8$ are on $L_1=0$ and the other $4$ on $L_2=0$,
where points in the pairs $\{P_1,P_2\}$, $\{P_3,P_4\}$, $\{P_5,P_6\}$,
$\{P_7,P_8\}$ are on different lines. Assume that $P_1$ is on $L_1$ and $P_2$
on $L_2$. Then $P_5$ is on $L_2$ since otherwise $L_1=0$ would be
$\{ x_3 = 0 \}$ but $\{ x_3 = 0 \}$ is not contained in $\bar{Q}=0$ by
eq.~\eqref{eq:ubis2}. Hence, we get
$$
\{L_1 = 0\} = \overline{P_1P_6}:\, (a_{21} - a_{22}) x_2 + (a_{21} - a_{23}) x_3 = 0.
$$
Now, it is easy to check that neither $P_3$ or $P_4$ are on $L_1=0$: for $P_3$ it is enough to check that $a_{22}-a_{23}\neq0$ and for $P_4$ that $a_{21}(a_{22} - a_{23}) \neq 0$, which is true by \cite[16.14]{weber}.

This gives a contradiction with the assumption $\bar{Q} = L_1\cdot L_2$. Hence, $\bar{Q}=0$ defines a non-singular conic.

\end{proof}
	


\section{Bitangents of a smooth plane quartic}  \label{sec:bitangents}
Let  $C/K$ be a smooth plane quartic curve given by $F=0$.  As $C$ is
canonically embedded in the projective plane, the canonical
divisors  on $C$ are the intersection of the
lines with $C$. In particular we can describe the $28$ odd theta characteristic divisors using bitangents. 

\begin{definition}
	A line $\beta$ is called a \emph{bitangent} of $C$ if the intersection divisor $(\beta \cdot C)$ is of the form $2P+2 Q$ for some not necessarily distinct points 
	$P,Q$ of $C$. The divisor $D_{\beta}:=P+Q$ is the odd theta characteristic divisor associated to $\beta$.
\end{definition}

\begin{remark}\label{rmk:construction}
	The bitangents of a curve can be computed by looking at the singular points of the dual curve, but this is rather expensive in terms of computations since the singularities also contains the tangents at inflexion points. A better approach is to work out the two algebraic conditions in $a$ and $b$ under which the form $F(x_1,a x_1+b,1)$ is a perfect square and to look for the solutions of the corresponding system. Of course, one has to take care of the bitangents which are not of the form $x_2=a x_1+b$.  
\end{remark}

\begin{definition}\label{def:asy}
	Let $S=\{\beta_i\}$ be a set of bitangents of $C$. The set $S$ is
	called \textit{azygetic} if there is no conic $Q$ such that the intersection divisor $Q \cdot C \geq D_{\beta_1}+D_{\beta_2}+D_{\beta_3}$ for $\beta_1,\beta_2,\beta_3 \in S$.
 An azygetic set of $7$ elements is called an \emph{Aronhold set}.
\end{definition}


Among the $\binom{28}{7}$ subsets of $7$ bitangents of $C$, 288 of them form an Aronhold set, see \cite[after Cor. 2.5]{grossharris}. Note that if one knows equations for the $28$ bitangents and the curve, a tedious but straighforward computation using the definition of azygeticness allow to exhibit an Aronhold set.\\

 We resume with the notation from the previous sections. Let us recall the following result \cite{riemann}.

\begin{theorem} \label{thm:Riemann}
	Let $C$ be a smooth plane quartic curve and $S=\{\beta_i\}_{i=1,\ldots,7}$ be an Aronhold set. After a linear change of variables, we may assume that the $\beta_i$ are given by the equations:
	$$
	\begin{array}{ccc}
	\beta_i : x_i=0 &	\beta_4 : x_1+x_2+x_3=0 & \beta_{4+i} : a'_{i1} x_1+a'_{i2} x_2+a'_{i3} x_{3}=0,
	\end{array}
	$$
	where $i\in\{1,2,3\}$. The coefficients $a'_{ij}$ are multiples of the $a_{ij}$ (defined in Section~\ref{sec:theta}), i.e. $a'_{ij}=\eta_ia_{ij}$ for some $\eta_i \ne 0$, that are determined, up to sign, by the linear system:
	$$
	\begin{bmatrix}
	\lambda_1{a'_{11}} & \lambda_2{a'_{21}} & \lambda_3{a'_{31}}\\
	\lambda_1{a'_{12}} & \lambda_2{a'_{22}} & \lambda_3{a'_{32}}\\
	\lambda_1{a'_{13}} & \lambda_2{a'_{23}} & \lambda_3{a'_{33}}\\
	\end{bmatrix}\begin{bmatrix}
	1/\eta_1^2 \\ 1/\eta_2^2 \\ 1/\eta_3^2
	\end{bmatrix}= \begin{bmatrix}
	-1\\-1\\-1
	\end{bmatrix},
	$$
	where $\lambda_1$, $\lambda_2$, $\lambda_3$ are given by
	$$
	\begin{bmatrix}
	\frac{1}{a'_{11}} & \frac{1}{a'_{21}} & \frac{1}{a'_{31}}\\
	\frac{1}{a'_{12}} & \frac{1}{a'_{22}} & \frac{1}{a'_{32}}\\
	\frac{1}{a'_{13}} & \frac{1}{a'_{23}} & \frac{1}{a'_{33}}\\
	\end{bmatrix}\begin{bmatrix}
	\lambda_1 \\ \lambda_2 \\ \lambda_3
	\end{bmatrix}= \begin{bmatrix}
	-1\\-1\\-1
	\end{bmatrix}.
	$$
	After the previous change of variables, the plane quartic $C$ is given by a {Riemann model}
	$$\left(x_1 u_1 + x_2 u_2 -x_3 u_3\right)^2-4 x_1 u_1 x_2 u_2 =0,$$
	where $u_1,u_2,u_3$ are  given as in Section~\ref{sec:theta} by
	$$\begin{cases}
	u_1 + u_2 + u_3 +x_1 +x_2+x_3 =0, \\
	\frac{u_1}{a_{i1}} + \frac{u_2}{a_{i2}} +\frac{u_3}{a_{i3}}+ a_{i1} x_1 + a_{i2}
	x_2+ a_{i3}
	x_3=0.
	\end{cases}$$
	Moreover, we can express all the bitangents for this model as:
	$$
	\begin{array}{ccc}
	\beta_i : x_i=0 &	\beta_4 : x_1+x_2+x_3=0 & \beta_{ij} : u_k=0 \\
	\beta_{4+i} : a_{i1} x_1+a_{i2} x_2+a_{i3} x_{3}=0  &	\beta_{i4}: u_i+x_j+x_k=0 & \beta_{i(4+l)}: 
	\frac{u_i}{a_{li}}+ a_{lj} x_j+ a_{lk} x_k=0 \\
	\end{array}$$
	$$\begin{array}{cc}
	\beta_{4(4+i)} : \frac{u_1}{a_{i1}(1- a_{i2} a_{i3})} +\frac{u_2}{a_{i2} (1-  a_{i3} a_{i1})} +
	\frac{u_3}{a_{i3}  (1- a_{i1} a_{i2})}=0 & \beta_{56} : \frac{u_1}{1-a_{32} a_{33}} +\frac{u_2}{1-  a_{33} a_{31}} +
	\frac{u_3}{1- a_{31} a_{32}}=0 \\
	\end{array}$$
	$$\begin{array}{cc}
	\beta_{57} : \frac{u_1}{1- a_{22} a_{23}} +\frac{u_2}{1-  a_{23} a_{21}} +
	\frac{u_3}{1- a_{21} a_{22}}=0 &
	\beta_{67} : \frac{u_1}{1- a_{12} a_{13}} +\frac{u_2}{1-  a_{13} a_{11}} +
	\frac{u_3}{1- a_{11} a_{12}}=0, \\
	\end{array}
	$$
	where $i,l\in\{1,2,3\}$ and $\{i,j,k\}=\{1,2,3\}$.
\end{theorem}

\begin{proposition}\label{prop:3cases}
Let assume that $C$ has potentially good hyperelliptic reduction. Given an Aronhold set $S$ as in Theorem \ref{thm:Riemann}, there exists a constant $v_0>0$ such that either
\begin{itemize}
	\item[(1)] two of the $a_{ij}$ (with the same value of $j$) have positive valuation $v_0$ and the rest have valuation equal to zero; 
	\item[(2)] or two of the $a_{ij}$ (with the same value of $j$) have negative valuation $-v_0$ and the rest have valuation equal to zero;
	\item[(3)] or all the $a_{ij}$ have valuation equal to zero.
\end{itemize}
\end{proposition}
\begin{proof}
	Look at the expressions for the coefficients $a_{ij}$ in terms of the theta constants in page $108$ formula $(11')$ in \cite{weber} and consider the different cases. For example, if the theta constant that has positive valuation $v_0$ is $\thetar\car{100}{001}$ as in the proof of Theorem \ref{Thm:charThetas}, we can see that all $a_{ij}$ have zero valuation except
	$$
	a_{11}=\frac{\thetar\car{100}{001}\thetar\car{000}{101}}{\thetar\car{101}{000}\thetar\car{001}{100}}\text{ and }a_{31}=\frac{\thetar\car{110}{110}\thetar\car{100}{001}}{\thetar\car{001}{100}\thetar\car{011}{110}}
	$$
	that have also positive valuation $v_0$, so we are in case $(1)$ of the Proposition.
\end{proof}

\begin{theorem}\label{thm:goodAronhold}
Let assume that $C$ has potentially good hyperelliptic reduction. Given an Aronhold set $S =\{\beta_i \}_{i=1,\ldots,7}$  we can construct another Aronhold set $\{\gamma_i \}_{i=1,\ldots,7}$ for which $v(a'_{11})=v(a'_{31})>0$. 
\end{theorem}

\begin{proof}
Notice that if we are already in the first case, then we may assume that $v(a_{1j})=v(a_{3j})=v_0$ for $v_0>0$ and some $j$ after permuting the set $\{\beta_1,\beta_{2},\beta_3\}$ of bitangents, and indeed, that $j=1$ after permuting the set $\{\beta_4,\beta_{5},\beta_6\}$.

If we are in the second case, and after a permutation of the sets $\{\beta_1,\beta_{2},\beta_3\}$ and $\{\beta_4,\beta_{5},\beta_6\}$, we may assume that $a_{11}$ and $a_{31}$ are the ones with negative valuations, that is, $\thetar\car{001}{100}$ is the theta constant with positive valuation equal to $v_0$. Then we can take the Aronhold system:
$$
\gamma_1=\beta_{23},\,\gamma_2=\beta_2,\,\gamma_3=\beta_3,\,\gamma_4=\beta_{14},
$$
$$
\gamma_5=\beta_{15},\,\gamma_6=\beta_{16},\,\gamma_6=\beta_{16}.
$$
So, $a'_{i1}=1/a_{i1}$ and $a'_{ij}=a_{ij}$ for $i\in\{1,2,3\}$ and $j\in\{2,3\}$, and we fall in the first case with $v(a'_{11})=v(a'_{31})=v_0$.

 In the third case we have to distinguish several cases. If any of the differences $1-a_{ij}a_{i(j+1)}$ has positive valuation, then after permutations of the sets $\{\beta_1,\beta_{2},\beta_3\}$ and $\{\beta_4,\beta_{5},\beta_6\}$, we can assume that $\thetar\car{100}{000}$ is the theta constant with positive valuation, then again by using Weber's formulas \cite[pp. 109]{weber}, we have that $\{\gamma_1,\ldots,\gamma_7 \}=\{\beta_{23},\,\beta_{13}\,\,\beta_{12},\,\beta_{4},\,\beta_{45},\,\beta_{46},\,\beta_{47}\}$ is an Aronhold system with $v(a'_{11})=v(a'_{31})=v_0$.
  If none of the differences $1-a_{ij}a_{i(j+1)}$ has positive valuation, then the differences $a_{ij}a_{i(j+1)}-a_{(i+1)j}a_{(i+1)(j+1)}$ have positive valuations. Hence, the Aronhold system $\{\gamma_1,\ldots,\gamma_7 \}=\{\beta_{5},\,\beta_{6}\,\,\beta_{7},\,\beta_{4},\,\beta_{1},\,\beta_{2},\,\beta_{3}\}$ falls in the first case.
\end{proof}



\section{The algorithm and an example}  \label{sec:computations}
In this section we start with  a smooth plane quartic $C/K:\,F=0$, where $K$ is a finite extension of $\mathbb{Q}_p$ with $p \ne 2$. The previous sections lead to the following algorithm to compute a stable model when the curve has potentially good hyperelliptic reduction.\\

\begin{tabular}{ll}
\hline 
& Computation of the stable model of $C$\\
\hline 
\texttt{Input}: &  an equation $F$ and a precision $e$ big enough.\\
\texttt{Output}:&  a stable model of $C$ when $C$ has potentially good hyperelliptic reduction \\
\hline
1 & Compute the set $B$ of bitangents of $C$ over a small extension of $K$ (using remark \ref{rmk:construction})\\
2 & Take any two bitangents $b_1,b_2$ and let $S=\{b_1,b_2\}$ \\
3 & \texttt{While} $\# S \ne 7$:\\
4& \hspace{1cm}  \texttt{For} $b \in B \setminus S$:\\
5& \hspace{2cm}  \texttt{if} $S \cup \{b\}$ is azygetic:\\
6 & \hspace{3cm}  $S=S \cup \{b\}$\\
7 & \hspace{2cm} \texttt{else} \\
8 & \hspace{3cm}  $B=B \setminus \{b\}$\\
9 & Send $4$ elements of $S$ by a change of variables to \\
&  $x_1=0$, $x_2=0$, $x_3=0$ and $x_1+x_2+x_3=0$ and denote by $S_0$  the resulting set \\
10 &  Normalize the coefficients of the elements in $S_0$ with Theorem \ref{thm:Riemann} \\
11 & Use Theorem \ref{thm:goodAronhold} and re-label the bitangents to get positive valuations for $a_{11}$ and $a_{31}$ \\ 
12 & Compute the corresponding Riemann model
\end{tabular}

\begin{remark} The azygetic test on line $(5)$ is perform by checking if the $6$ points of intersection of $3$ of the bitangents in $S \cup \{b\}$ with the curve lie on a conic. Notice that as we are over a non-exact field, this step is harder than over a number field.
\end{remark}


\begin{example}\label{ex:Bas}
Let us consider the plane smooth quartic $C/\Q : F=0$ where
$$F=(x_2+x_3) x_1^3-(2 x_2^2 + x_3 x_2) x_1^2+(x_2^3-x_3 x_2^2+2 x_3^2 x_2-x_3^3) x_1-(2 x_3^2 x_2^2-3 x_3^3 x_2).$$
This curve is an equation of $X_{ns}(13) \simeq X_s(13) \simeq X_0^{+}(169)$ which is studied in \cite[Cor.6.8]{BDMTV17}. Using a general result from~\cite{edixhoven90}, they prove that this curve has good reduction everywhere and potentially good reduction at 13 after a ramified extension of degree 84. Using the characterizations in \cite{LLLR} in terms of the valuations of the Dixmier-Ohno invariants, it is straightforward to see that this is indeed the case and that the stable reduction at 13 is the hyperelliptic curve $y^2=z(x^7-z^7)$. We apply our algorithm to find an equation for the stable model at $13$.

 Starting from $F$, we get 
  the following Aronhold system  modulo $13^2$:
  $$
  \begin{array}{cccc}
    x_1 = 0,\  x_2 = 0,\  x_3 = 0,\  x_1+x_2+x_3 = 0,\\
  \end{array}
  $$
  $$
\begin{array}{ccc}
    a_{11}\,x_1+a_{12}\,x_2+a_{13}\,x_3 = 0, & 
    a _{21}\,x_1+a _{22}\,x_2+a_{23}\,x_3 = 0, &
    a_{31}\,x_1+a_{32}\,x_2+a_{33}\,x_3 = 0,
  \end{array}
  $$
with
\begin{footnotesize}
  \begin{displaymath}
    \setlength{\arraycolsep}{1pt}
  \begin{array}{rcl}
    a_{11} &=& 73\,\pi^6+(156\,\tau+27)\,\pi^5+(4\,\tau+96)\,\pi^4+(28\,\tau+57)\,\pi^3+(8\,\tau+10)\,\pi^2+(78\,\tau+101)\,\pi+38\,\tau+26,\\
    a_{12} &=&  24\,\tau+94)\,\pi^6+(13\,\tau+137)\,\pi^5+(32\,\tau+8)\,\pi^4+(54\,\tau+76)\,\pi^3+(47\,\tau+166)\,\pi^2+(143\,\tau+75)\,\pi+77\,\tau+13,\\
    a_{13} &=&  (153\,\tau+57)\,\pi^6+(167\,\tau+67)\,\pi^5+(29\,\tau+35)\,\pi^4+(83\,\tau+10)\,\pi^3+(74\,\tau+94)\,\pi^2+(161\,\tau+154)\,\pi+56\,\tau+74,\\[0.2cm]
    a_{21} &=&  (128\,\tau+140)\,\pi^6+(45\,\tau+69)\,\pi^5+(110\,\tau+12)\,\pi^4+(112\,\tau+13)\,\pi^3+(60\,\tau+90)\,\pi^2+(10\,\tau+110)\,\pi+150\,\tau+159,\\
    a_{22}&=&  (120\,\tau+40)\,\pi^6+(77\,\tau+106)\,\pi^5+(36\,\tau+50)\,\pi^4+(89\,\tau+100)\,\pi^3+(130\,\tau+120)\,\pi^2+(64\,\tau+64)\,\pi+17\,\tau+73,\\
    a_{23}&=&   (142\,\tau+92)\,\pi^6+(44\,\tau+26)\,\pi^5+(156\,\tau+145)\,\pi^4+(88\,\tau+147)\,\pi^3+(72\,\tau+109)\,\pi^2+(104\,\tau+93)\,\pi+100\,\tau+69,\\[0.2cm]
    a_{31} &=&  (141\,\tau+27)\,\pi^6+(121\,\tau+39)\,\pi^5+(6\,\tau+103)\,\pi^4+(\tau+141)\,\pi^3+(51\,\tau+74)\,\pi^2+(79\,\tau+32)\,\pi+129\,\tau+39,\\
    a_{32} &=&  (39\,\tau+11)\,\pi^6+(8\,\tau+97)\,\pi^5+(155\,\tau+4)\,\pi^4+(92\,\tau+36)\,\pi^3+(103\,\tau+100)\,\pi^2+(66\,\tau+45)\,\pi+77\,\tau+13,\\
    a_{33} &=&  (154\,\tau+28)\,\pi^6+(77\,\tau+126)\,\pi^5+(124\,\tau+5)\,\pi^4+(151\,\tau+86)\,\pi^3+(2\,\tau+149)\,\pi^2+(141\,\tau+8)\,\pi+33\,\tau+105\,.\\
  \end{array}
\end{displaymath}
\end{footnotesize}
where
\begin{displaymath}
  \tau^2 + 12\,\tau + 2 = 0 + O(13^2)
  \text{ and }
  \pi^7  +\,13\cdot6\,( {\pi}^{2}+5\,\pi+8 )  ( {\pi}^{4}+2\,{\pi}^{3}+2\,{\pi}^{2}+11\,\pi+12 ) = 0 + O(13^2)\,.
\end{displaymath}

A \PHM~model modulo $13^2$ is then $Q^2 + \pi G=0$ where
\begin{displaymath}
  Q = (9\,\tau+3)\,x_1^2+(4\,\tau+2)\,x_1\,x_2+2\,x_1\,x_3+(8\,\tau+12)\,x_2^2+x_2\,x_3+O(13^2)
\end{displaymath}
and
\begin{footnotesize}
  \begin{displaymath}
    \setlength{\arraycolsep}{1pt}
  \begin{array}{ll}
  G = &O(13^2)+\\
    &((12\,\tau+11)\,\pi^6+(18\,\tau+74)\,\pi^5+(124\,\tau+120)\,\pi^4+(40\,\tau+94)\,\pi^3+(45\,\tau+105)\,\pi^2+(8\,\tau+61)\,\pi+145\,\tau+117)\,x_1^4+\\
    &((6\,\tau+122)\,\pi^6+(68\,\tau+30)\,\pi^5+(78\,\tau+152)\,\pi^4+(10\,\tau+142)\,\pi^3+(106\,\tau+38)\,\pi^2+(168\,\tau+47)\,\pi+162\,\tau+86)\,x_1^3\,x_2+\\
    &((23\,\tau+135)\,\pi^6+(103\,\tau+9)\,\pi^5+(6\,\tau+160)\,\pi^4+(20\,\tau+27)\,\pi^3+(2\,\tau+155)\,\pi^2+(146\,\tau+106)\,\pi+69\,\tau+37)\,x_1^3\,x_3+\\
    &((106\,\tau+65)\,\pi^6+(81\,\tau+26)\,\pi^5+(92\,\tau+8)\,\pi^4+(22\,\tau+99)\,\pi^3+(6\,\tau+133)\,\pi^2+(82\,\tau+76)\,\pi+46\,\tau+104)\,x_1^2\,x_2^2+\\
    &((116\,\tau+102)\,\pi^6+(83\,\tau+78)\,\pi^5+(60\,\tau+152)\,\pi^4+(97\,\tau+129)\,\pi^3+(106\,\tau+44)\,\pi^2+(92\,\tau+165)\,\pi+89\,\tau+66)\,x_1^2\,x_2\,x_3+\\
    &((135\,\tau+10)\,\pi^6+(125\,\tau+97)\,\pi^5+(18\,\tau+93)\,\pi^4+(104\,\tau+92)\,\pi^3+(130\,\tau+151)\,\pi^2+(156\,\tau+101)\,\pi+104+78\,\tau)\,x_1^2\,x_3^2+\\
    &((88\,\tau+77)\,\pi^6+(73\,\tau+109)\,\pi^5+(113\,\tau+68)\,\pi^4+(13\,\tau+12)\,\pi^3+(96\,\tau+96)\,\pi^2+(56\,\tau+125)\,\pi+82\,\tau+15)\,x_1\,x_2^3+\\
    &((53\,\tau+124)\,\pi^6+(112\,\tau+64)\,\pi^5+(123\,\tau+109)\,\pi^4+(145\,\tau+123)\,\pi^3+(77\,\tau+66)\,\pi^2+(18\,\tau+85)\,\pi+29\,\tau+96)\,x_1\,x_2^2\,x_3+\\
    &((8\,\tau+112)\,\pi^6+(20\,\tau+45)\,\pi^5+(73\,\tau+73)\,\pi^4+(96\,\tau+58)\,\pi^3+(138\,\tau+79)\,\pi^2+(130\,\tau+158)\,\pi+156)\,x_1\,x_2\,x_3^2+\\
    &((61\,\tau+84)\,\pi^6+(102\,\tau+138)\,\pi^5+(53\,\tau+52)\,\pi^4+(141\,\tau+40)\,\pi^3+(106\,\tau+77)\,\pi^2+143\,\pi+143\,\tau+130)\,x_1\,x_3^3+\\
    &((42\,\tau+92)\,\pi^6+(30\,\tau+144)\,\pi^5+(66\,\tau+40)\,\pi^4+(39\,\tau+156)\,\pi^3+(60\,\tau+104)\,\pi^2+(56\,\tau+48)\,\pi+53\,\tau+75)\,x_2^4+\\
    &((161\,\tau+125)\,\pi^6+(154\,\tau+139)\,\pi^5+(39\,\tau+40)\,\pi^4+(94\,\tau+40)\,\pi^3+(82\,\tau+106)\,\pi^2+(111\,\tau+156)\,\pi+35\,\tau+119)\,x_2^3\,x_3+\\
    &((164\,\tau+105)\,\pi^6+(22\,\tau+24)\,\pi^5+(82\,\tau+58)\,\pi^4+(105\,\tau+147)\,\pi^3+(116\,\tau+126)\,\pi^2+(143\,\tau+135)\,\pi+117\,\tau+13)\,x_2^2\,x_3^2+\\
    &((134\,\tau+47)\,\pi^6+(114\,\tau+145)\,\pi^5+(164\,\tau+64)\,\pi^4+(\tau+6)\,\pi^3+(38\,\tau+20)\,\pi^2+(78\,\tau+13)\,\pi+78\,\tau+52)\,x_2\,x_3^3+\\
    &((117\,\tau+165)\,\pi^6+(65\,\tau+28)\,\pi^5+(65\,\tau+52)\,\pi^4+(104\,\tau+104)\,\pi^3+13\,\tau\,\pi^2+(52\,\tau+13)\,\pi+78\,\tau+13)\,x_3^4\,.\\
  \end{array}
\end{displaymath}
\end{footnotesize}
\end{example}
The Dixmier-Ohno invariants of this model modulo $13^2$ coincide with the Dixmier-Ohno invariants of the input curve. This proves the correctness of the output with the given precision. 

\begin{remark}
{In order to automatized completely this procedure, one would need packages to work with algebraic systems over the $p$-adic in a transparent way with control of the errors. Although there are results and fast progress on this topic \cite{Lebreton12, CRV14, CRV15, CRV18, Lebreton15}, such functionalities are still not fully implemented.}
\end{remark}




\bibliographystyle{alphaabbr}

\bibliography{synthbib}

\end{document}